\documentclass{article}

\usepackage{amsthm}
\usepackage{amssymb}
\usepackage{amsmath}
\usepackage{cite}

\usepackage{a4wide}

\newtheorem{theorem}{Theorem}[section]

\newtheorem{lemma}[theorem]{Lemma}
\theoremstyle{remark}
\newtheorem{remark}{Remark}[section]

\theoremstyle{definition}
\newtheorem{definition}{Definition}[section]
\theoremstyle{definition}

\begin{document}

\title{Existence and uniqueness of a positive solution\\
to generalized nonlocal thermistor problems\\
with fractional-order derivatives\thanks{Submitted 17-Jul-2011;
revised 09-Oct-2011; accepted 21-Oct-2011; for publication in the journal
\emph{Differential Equations \& Applications} ({\tt http://dea.ele-math.com}).}}

\author{Moulay Rchid Sidi Ammi\\
Group of Mathematical and Numerical Analysis of PDEs and Applications (AMNEA)\\
Department of Mathematics, Moulay Ismail University\\
Faculty of Science and Technology\\
B.P. 509, Errachidia, Morocco\\
\texttt{sidiammi@ua.pt} \\[0.3cm]
\and
Delfim F. M. Torres\\
Center for Research and Development in Mathematics and Applications (CIDMA)\\
Department of Mathematics, University of Aveiro\\
Campus Universit\'{a}rio de Santiago\\
3810-193 Aveiro, Portugal\\
\texttt{delfim@ua.pt}}

\date{}

\maketitle

% -------------------------------------------------

\begin{abstract}
In this work we study a generalized nonlocal thermistor problem
with fractional-order Riemann--Liouville derivative.
Making use of fixed-point theory, we obtain existence
and uniqueness of a positive solution.
\end{abstract}

% -------------------------------------------------

\smallskip


\smallskip

% -------------------------------------------------

\smallskip


\medskip

% -------------------------------------------------

\section{Introduction}

Joule heating is generated by the resistance of materials to electrical current
and is present in any electrical conductor operating at normal temperatures.
The heating of such conductors has undesirable side effects.
Problems dealing with the combined heat and current flows were considered
in \cite{cim,MR1215637,kw,mac}, where various aspects of the so-called thermistor
problem were analyzed. The mathematical model of the nonlocal
steady thermistor problem has the form
\begin{equation}
\label{eq1}
\triangle u = \frac{\lambda f(u)}{\left(
\int_{\Omega} f(u)\, dx\right)^{2}},
\end{equation}
where $\triangle$ is the Laplacian with respect to the spacial
variables. Such problems arise in many applications, for instance,
in studying the heat transfer in a resistor device
whose electrical conductivity $f$ is strongly
dependent on the temperature $u$. The equation \eqref{eq1}
describes the diffusion of the temperature with the presence of a
nonlocal term as a result of Joule effect. Constant $\lambda$
is a dimensionless parameter, which can
be identified with the square of the applied potential difference
at the ends of the conductor. Function $u$ represents the temperature generated
by the electric current flowing through a conductor. For more description, we refer to
\cite{lac2,tza}. A deep discussion about the history of thermistors,
and more detailed accounts of their advantages and applications to
industry, can be found in \cite{kw,mac}. In \cite{ac} Antontsev and Chipot
studied existence and regularity of weak solutions to the
thermistor problem under the condition that the electrical
conductivity $f(u)$ is bounded.

Fractional differential equations are a generalization of ordinary
differential equations and integration to arbitrary noninteger orders.
The origin of fractional calculus goes back to Newton and Leibniz
in the seventieth century. In recent years there has been a great deal
of interest in fractional differential equations. They provide a powerful
tool for modeling and solving various problems in various fields:
physics, mechanics, engineering, electrochemistry, economics, visco-elasticity,
feedback amplifiers, electrical circuits and fractional multipoles
--- see, for example, \cite{72,73,saba,samko,75,7} and references therein.
Using fixed point theorems, like Shauder's fixed point theorem,
and Banach's contraction mapping principle,
many results of existence have been obtained
to linear and nonlinear equations and, more recently,
to fractional derivative equations. The interested reader
can see \cite{agarwal1,agarwal}. For a physical meaning
to the initial conditions of fractional differential
equations with Riemann--Liouville derivatives
we refer to \cite{gian,MyID:181,pod1}.

Our main concern in this paper is to prove existence and uniqueness
of solution to a general fractional order nonlocal thermistor problem of the form
\begin{equation}
\label{eq2}
\begin{gathered}
 D^{2 \alpha} u = \frac{\lambda f(u)}{(
\int_{0}^{T} f(u)\, dx)^{2}}+ h(t) \, , \quad  t \in  (0, T)  \, , \\
I^{\beta} u(t)|_{t=0} = 0 ,  \quad    \forall \beta \in  (0,1],
\end{gathered}
\end{equation}
under suitable conditions on $f$ and $h$ (see Theorem~\ref{thm1}).
We assume that $T$ is a fixed positive real
and $ \alpha > 0$ a parameter describing the order of the fractional derivative.
In the literature we may find a great number of definitions of fractional derivatives.
In this paper, the fractional derivative is considered in the Riemann--Liouville sense.
In the case $\alpha =1$ and $h=0$, the fractional equation \eqref{eq2}
becomes the one-dimensional nonlocal steady state thermistor problem.
The values of $0< \alpha <\frac{1}{2}$ correspond to intermediate processes.
We further prove the boundedness of $u$ (see Theorem~\ref{thm:bnd}),
which is of considerable importance from a practical and physical point of view:
it is interesting to keep the temperature from exceeding some extremal values
that may damage the conductor.

% -------------------------------------------------

\section{Preliminaries}

In this section, we give some basic definitions and preliminary
facts that are used further in the paper.
Let $0< \alpha <\frac{1}{2}$ and
$X = \left( C([0,T]) , \| \cdot \| \right)$,
where $C([0,T])$ is the space  of all continuous functions on $[0,T]$.
For $ x \in C([0,T])$, define the norm
$$
\| x \| =  \sup_{t\in [0,T]} \{ e^{-Nt} |x(t)| \},
$$
which is equivalent to the standard supremum norm for $f \in
C([0,T])$. It is used in literature in many papers, see for
example \cite{sayed}. The use of this norm is technical and
allow us to simplify the integral calculus. By $L^{1}([0,T],
\mathbb{R})$ we denote the set of Lebesgue integrable functions on
$[0,T]$. Throughout the text $c$ denote constants which may change
at each occurrence. As in \cite{ac}, we consider that the
electrical conductivity is bounded. We now assume the following
assumptions:
\begin{itemize}

\item[(H1)] $f: \mathbb{R}^{+}\rightarrow \mathbb{R}^{+}$ is a Lipschitz continuous
function with Lipschitz constant $L_{f}$ such that $c_{1} \leq f(u) \leq c_{2}$,
with $c_{1}, c_{2}$ two positive constants;

\item[(H2)] $h$ is continuous on $(0, T)$ with $h \in L^{\infty}(0, T)$.

\end{itemize}

\begin{definition}
\label{def1} \rm
The fractional (arbitrary) integral
of order $ \alpha \in \mathbb{R}^+$
of a function $f \in L_1[a,b] $ is defined by
\[
I^{\alpha}_a
 f(t) = \int_a^t \frac{(t - s)^{ \alpha-1}}{\Gamma( \alpha)} f(s)\, ds,
\]
where $\Gamma$ is the gamma function
(see, \textrm{e.g.}, \cite{4,pod1,5,samko}).
For $a = 0$ we put $I^{\alpha} := I^{\alpha}_0$.
\end{definition}

\begin{remark}
For $f, g \in L_1[a,b]$ one has
\[
I^{ \alpha}_a (f(t)+g(t))
= I^{ \alpha}_a f(t) + I^{ \alpha}_a g(t).
 \]
Note also that $I^{ \alpha}f(t) \in C(\mathbb{R}^{+})$
for $f \in C(\mathbb{R}^{+})$ and,
moreover, $I^{ \alpha}f(0)=0$.
\end{remark}

\begin{definition}[see, \textrm{e.g.}, \cite{4,pod1,5,samko}]
\label{def2}
\rm
The Riemann--Liouville fractional
(arbitrary) derivative of order $ \alpha \in (n-1,n)$, $n \in \mathbb{N}$,
of function $f$ is defined by
\[
D^{ \alpha}_a f(t) = \frac{d^n}{dt^n} I^{n- \alpha}_a f(t)
= \frac{1}{\Gamma(n- \alpha)}  \left(\frac{d}{dt}\right)^n
\int_a^t (t - s)^{n- \alpha-1} f(s) ds ,
\quad   t \in [a,b].
\]
\end{definition}

% -----------------------------------------

\section{Main Results}

Our main result asserts existence of a unique solution
to \eqref{eq2} on $C(\mathbb{R}^{+})$ of the form
\begin{equation}
\label{eq4}
\begin{split}
u(t) &= I^{2\alpha} \left\{\frac{\lambda f(u)}{( \int_{0}^{T} f(u)\,
dx)^{2}}+ h(t) \right\}\\
&= \int_0^t\frac{(t-s)^{2 \alpha-1}}{\Gamma(2
\alpha)}  \left\{\frac{\lambda f(u)}{( \int_{0}^{T} f(u)\, dx)^{2}}
+ h(s) \right\}  ds \, .
\end{split}
\end{equation}

% ----------------------

\subsection{Existence and Uniqueness}

We begin by proving the equivalence between \eqref{eq2} and \eqref{eq4}
on the space $C(\mathbb{R}^{+})$. This restriction of the space of functions
allows to exclude from the proof a stationary function
with Riemann-Liouville derivative of order $2\alpha$ equal
to $d \cdot t^{2\alpha-1}$, $d\in \mathbb{R}$, which
belongs to the space $C_{1-2\alpha}[0, T]$
of continuous weighted functions.

\begin{lemma}
\label{m:lema} Suppose that $\alpha \in (0, \frac{1}{2})$. Then
the nonlocal problem \eqref{eq2} is equivalent to the integral
equation \eqref{eq4} on the space $C(\mathbb{R}^{+})$.
\end{lemma}

\begin{proof}
First we prove that \eqref{eq2} implies \eqref{eq4}.
For $t>0$ equation \eqref{eq2} can be written as
\[
\frac{d}{dt} I^{1-2 \alpha} u(t) =\frac{\lambda f(u)}{\left(
\int_{0}^{T} f(u)\, dx\right)^{2}} + h(t).
\]
Integrating both sides of the above equation, we obtain
\[
I^{1-2 \alpha} u(t)
- I^{1-2 \alpha} u(t)|_{t=0} = \int_0^t
\left\{\frac{\lambda f(u)}{\left(
\int_{0}^{T} f(u)\, dx\right)^{2}}+ h(s)\right\} \,ds.
\]
Since $0< 1-2 \alpha < 1$,
\[
I^{1-2 \alpha} u(t) = \int_0^t
\left\{\frac{\lambda f(u)}{\left(
\int_{0}^{T} f(u)\, dx\right)^{2}}+ h(s)\right\}\, ds\,.
\]
Applying the operator $I^{2 \alpha}$ to both sides, we get
\[
I u(t) = I^{2 \alpha+1} \left\{\frac{\lambda f(u)}{\left(
\int_{0}^{T} f(u)\, dx\right)^{2}}+ h(t) \right\}.
\]
Differentiating both sides,
\[
u(t) = I^{2 \alpha} \left\{\frac{\lambda f(u)}{\left(
\int_{0}^{T} f(u)\, dx\right)^{2}}+ h(t) \right\}.
\]
Let us now prove that \eqref{eq4} implies \eqref{eq2}.
Since $u \in C(I)$ and $I^{1-2 \alpha}u(t) \in C(I)$,
applying the operator $I^{1-2 \alpha}$ to both sides
of \eqref{eq4} one obtains
\begin{align*}
I^{1-2 \alpha} u(t)
&= I^{1-2 \alpha} I^{2 \alpha} \left\{\frac{\lambda f(u)}{\left(
\int_{0}^{T} f(u)\, dx\right)^{2}}+ h(t)\right\}\\
&= I\left\{\frac{\lambda f(u)}{\left(
\int_{0}^{T} f(u)\, dx\right)^{2}}+ h(t) \right\}\,.
\end{align*}
Differentiating both sides of the above equality,
\[
D I^{1-2 \alpha} u(t) = D I \left\{\frac{\lambda f(u)}{\left(
\int_{0}^{T} f(u)\, dx\right)^{2}}+ h(t)\right\}.
\]
Then,
\[
D^{2 \alpha} u(t) = \frac{\lambda f(u)}{\left(
\int_{0}^{T} f(u)\, dx\right)^{2}}+ h(t),  \quad t > 0.
\]
\end{proof}

\begin{theorem}
\label{thm1}
Let $f$ and $h$ satisfy hypotheses $(H1)$ and $(H2)$.
Then there exists a unique solution $u \in X $ of \eqref{eq2} for
all $0 < \lambda < \frac{N^{2 \alpha}}{L_{f}
\left(\frac{1}{\left(c_{1}T\right)^{2}}
+ \frac{2c_{2}^{2}T}{\left(c_{1}T\right)^{4}} e^{NT}\right)}$.
\end{theorem}

\begin{proof}
Let $F : X  \to  X $ be defined by
\[
Fu= I^{2 \alpha} \left\{ \frac{\lambda f(u)}{\left(
\int_{0}^{T} f(u)\, dx\right)^{2}}+ h(t) \right\}.
\]
Then,
\begin{equation}
\begin{split}
\label{eq5}
|Fu-Fv| &= \left|I^{2 \alpha} \left\{
\frac{\lambda f(u)}{\left( \int_{0}^{T} f(u)\, dx\right)^{2}}
- \frac{\lambda f(v)}{\left( \int_{0}^{T} f(v)\, dx\right)^{2}} \right\} \right|\\
&= \left|I^{2 \alpha} \left\{
\frac{\lambda}{\left( \int_{0}^{T} f(u)\, dx\right)^{2}}(f(u)-f(v))
+\lambda f(v) \left( \frac{1}{\left( \int_{0}^{T} f(u)\, dx\right)^{2}}
- \frac{1}{\left( \int_{0}^{T} f(v)\, dx\right)^{2}} \right) \right\} \right|  \\
&\leq \left|I^{2 \alpha} \left\{
\frac{\lambda}{\left( \int_{0}^{T} f(u)\, dx\right)^{2}}(f(u)-f(v))
\right\} \right| + \left|I^{2 \alpha}
\left\{\lambda f(v) \left(  \frac{1}{\left( \int_{0}^{T} f(u)\, dx\right)^{2}}
- \frac{1}{\left( \int_{0}^{T} f(v)\, dx\right)^{2}} \right) \right\} \right|.
\end{split}
\end{equation}
We estimate each term on the right hand side of \eqref{eq5} separately.
Using then the fact that $f$ is Lipshitzian, we have
\begin{equation}
\begin{split}
\label{eq6}
\left|I^{2 \alpha} \left\{  \frac{\lambda}{
\left( \int_{0}^{T} f(u)\, dx\right)^{2}}(f(u)-f(v)) \right\} \right|
&\leq \frac{1}{(c_{1}T)^{2}} \lambda I^{2 \alpha} \{ |f(u)-f(v)| \}\\
&\leq \frac{1}{(c_{1}T)^{2}}\lambda L_{f} I^{2 \alpha} \{ |u-v| \}\\
&= \frac{1}{(c_{1}T)^{2}} \lambda L_{f} \int_0^t\frac{(t-s)^{2
\alpha-1}}{\Gamma(2 \alpha)} |u(s)-v(s)| ds.
\end{split}
\end{equation}
Since
$$
\int_0^{Nt}\frac{r^{2 \alpha-1}}{\Gamma(2 \alpha)} e^{-r}  dr
\leq \frac{1}{\Gamma(2 \alpha)} \int_0^{+\infty} r^{2 \alpha-1}e^{-r}  dr
= \frac{\Gamma(2 \alpha)}{\Gamma(2 \alpha)}=1,
$$
it follows from \eqref{eq6} that
\begin{equation*}
\begin{split}
e^{-Nt} &\left|I^{2 \alpha} \left\{
\frac{\lambda}{( \int_{0}^{T} f(u)\, dx)^{2}}(f(u)-f(v)) \right\} \right| \\
&\leq \frac{1}{(c_{1}T)^{2}} \lambda L_{f} e^{-Nt}
\int_0^t\frac{(t-s)^{2 \alpha-1}}{\Gamma(2 \alpha)} |u(s)-v(s)| ds \\
&\leq \frac{1}{(c_{1}T)^{2}} \lambda L_{f} \int_0^t\frac{(t-s)^{2
\alpha-1}}{\Gamma(2 \alpha)}
e^{-N(t-s)} e^{-Ns}|u(s)-v(s)| ds \\
&\leq\frac{1}{(c_{1}T)^{2}} \lambda L_{f}
\sup_t\{e^{-Nt}|u(t)-v(t)|\}
\int_0^t\frac{(t-s)^{2 \alpha-1}}{\Gamma(2 \alpha)} e^{-N(t-s)} ds \\
&\leq \frac{1}{(c_{1}T)^{2}} \lambda L_{f} \|u-v\|
\int_0^t\frac{(t-s)^{2 \alpha-1}}{\Gamma(2 \alpha)} e^{-N(t-s)} ds\\
&\leq \frac{1}{(c_{1}T)^{2}} \lambda L_{f} \|u-v\| \frac{1}{N^{2
\alpha}}\int_0^{Nt}
\frac{r^{2 \alpha-1}}{\Gamma(2 \alpha)} e^{-r} dr
\leq \frac{\frac{1}{(c_{1}T)^{2}} \lambda L_{f}}{N^{2 \alpha}}
\|u-v\|.
\end{split}
\end{equation*}
On the other hand, similar arguments as above yield to
\begin{equation}
\begin{split}
\label{eq8}
& \left|I^{2 \alpha} \left\{
\lambda f(v) \left( \frac{1}{\left( \int_{0}^{T} f(u)\, dx\right)^{2}}
-  \frac{1}{\left(\int_{0}^{T} f(v)\, dx\right)^{2}} \right) \right\} \right| \\
&= \left|I^{2 \alpha} \left\{  \frac{\lambda f(v)}{\left( \int_{0}^{T} f(u)\, dx\right)^{2}
\left( \int_{0}^{T} f(v)\, dx\right)^{2}} \left(  \left( \int_{0}^{T} f(u)\, dx\right)^{2}
-\left( \int_{0}^{T} f(v)\, dx\right)^{2} \right)\right\}\right|\\
&\leq  \frac{c_{2}}{(c_{1}T)^{4}} \lambda \left|I^{2 \alpha}
\left\{ \left( \int_{0}^{T} f(u)\, dx\right)^{2}
-\left( \int_{0}^{T} f(v)\, dx\right)^{2} \right\} \right|\\
&\leq  \frac{c_{2}}{(c_{1}T)^{4}}  \lambda \left|I^{2 \alpha}
\left\{  \left( \int_{0}^{T} (f(u)-f(v))\, dx\right)
\left( \int_{0}^{T} (f(u)+f(v))\, dx\right)  \right\}\right| \\
&\leq \frac{2c_{2}^{2}T}{(c_{1}T)^{4}}
\lambda I^{2 \alpha} \left\{ \int_{0}^{T} |f(u)-f(v)|\, dx \right\}\\
&\leq \frac{2c_{2}^{2}T}{(c_{1}T)^{4}}
\lambda L_{f} I^{2 \alpha}
\left\{ \int_{0}^{T} |u(x)-v(x)|\, dx \right\} \\
&\leq \frac{2c_{2}^{2}T}{(c_{1}T)^{4}}\lambda L_{f} \int_{0}^{t}
\frac{(t-s)^{2 \alpha-1}}{\Gamma(2 \alpha)}  \left( \int_{0}^{T}
|u(x)-v(x)|\, dx \right) ds.
\end{split}
\end{equation}
Then,
\begin{equation}
\label{eq9}
\begin{split}
e^{-Nt} & \left|I^{2 \alpha} \left\{
\lambda f(v) \left(  \frac{1}{( \int_{0}^{T} f(u)\, dx)^{2}}
- \frac{1}{\left( \int_{0}^{T} f(v)\, dx\right)^{2}} \right) \right\} \right|\\
&\leq  \frac{2c_{2}^{2}T}{(c_{1}T)^{4}} \lambda L_{f} \int_{0}^{t}
\frac{(t-s)^{2 \alpha-1}}{\Gamma(2 \alpha)}
\left( \int_{0}^{T} e^{-N(t-x)} e^{-Nx}|u(x)-v(x)|\, dx \right) ds\\
&\leq  \frac{2c_{2}^{2}T}{(c_{1}T)^{4}}\lambda L_{f}
\sup_t\{e^{-Nt}|u(t)-v(t)|\} \int_{0}^{t} \frac{(t-s)^{2
\alpha-1}}{\Gamma(2 \alpha)}
\left( \int_{0}^{T} e^{-N(t-x)} \, dx \right) ds \\
&\leq   \frac{2c_{2}^{2}T}{(c_{1}T)^{4}} \lambda L_{f}
\sup_t\left\{e^{-Nt}|u(t)-v(t)|\right\} \int_{0}^{t}
\frac{(t-s)^{2 \alpha-1}}{\Gamma(2 \alpha)}
e^{-Nt} \left( \frac{1}{N}(e^{NT}-1) \right) ds \\
&\leq  \frac{2c_{2}^{2}T}{(c_{1}T)^{4}} \lambda L_{f}\|u-v\|
\int_{0}^{t} \frac{(t-s)^{2 \alpha-1}}{\Gamma(2 \alpha)}e^{-Nt}
\left( \frac{1}{N}(e^{NT}-1) \right) ds \\
&\leq \frac{ \frac{2c_{2}^{2}T}{(c_{1}T)^{4}} e^{NT} \lambda
L_{f}}{N}\|u-v\|
\int_{0}^{t} \frac{(t-s)^{2 \alpha-1}}{\Gamma(2 \alpha)}e^{-Nt}  ds\\
&\leq \frac{ \frac{2c_{2}^{2}T}{(c_{1}T)^{4}}e^{NT} \lambda
L_{f}}{N}\|u-v\|
\int_{0}^{t} \frac{(t-s)^{2 \alpha-1}}{\Gamma(2 \alpha)}e^{-N(t-s)}e^{-Ns}  ds\\
&\leq \frac{ \frac{2c_{2}^{2}T}{(c_{1}T)^{4}} e^{NT} \lambda
L_{f}}{N}\|u-v\| \int_{0}^{t}
\frac{(t-s)^{2 \alpha-1}}{\Gamma(2 \alpha)}e^{-N(t-s)}  ds\\
&\leq \frac{ \frac{2c_{2}^{2}T}{(c_{1}T)^{4}} e^{NT} \lambda L_{f}}{N^{2 \alpha +1}}\|u-v\| \\
&\leq \frac{ \frac{2c_{2}^{2}T}{(c_{1}T)^{4}} e^{NT} \lambda
L_{f}}{N^{2 \alpha}}\|u-v\|.
\end{split}
\end{equation}
Gathering \eqref{eq5}--\eqref{eq9}, we get
\begin{equation*}
e^{-Nt} |Fu-Fv| \leq  \left(\frac{1}{(c_{1}T)^{2}}+
\frac{2c_{2}^{2}T}{(c_{1}T)^{4}} e^{NT}\right) \frac{\lambda
L_{f}}{N^{2 \alpha}}\|u-v\|.
\end{equation*}
Then, we have
\begin{equation*}
 \|Fu-Fv \| \leq  \left(\frac{1}{(c_{1}T)^{2}}+
\frac{2c_{2}^{2}T}{(c_{1}T)^{4}} e^{NT}\right)\frac{ \lambda
L_{f}}{N^{2 \alpha}}\|u-v\|.
\end{equation*}
Choosing $\lambda >0$ such that $\left(\frac{1}{(c_{1}T)^{2}}
+ \frac{2c_{2}^{2}T}{(c_{1}T)^{4}} e^{NT}\right)\frac{\lambda
L_{f}}{N^{2 \alpha}} < 1$, the map $ F : X \to X $ is a
contraction and it has a fixed point $u=Fu$. Hence, there exists a
unique $u \in X$ that is the solution to the integral equation
\eqref{eq4}. The result follows from Lemma~\ref{m:lema}.
\end{proof}

% ----------------------

\subsection{Boundedness}

We now show that the condition that the electrical conductivity $f(u)$ is bounded
(hypothesis (H1)) allows to assert boundedness of $u$.

\begin{theorem}
\label{thm:bnd} Under hypotheses $(H1)$ and $(H2)$ and $\lambda
> 0$, if $u$ is the solution of \eqref{eq4}, then
\[
\|u \| \leq  \frac{\left(
\frac{\lambda}{(c_{1}T)^{2}}f(0)+h_{\infty}\right)}{N^{2 \alpha}}
e^{\frac{\lambda L_{f}}{(c_{1}TN^{\alpha})^{2}}}.
\]
\end{theorem}

\begin{proof}
One has
\begin{equation*}
\begin{split}
|u(t)| &\leq  I^{2 \alpha} \left\{ \frac{\lambda |f(u)|}{(
\int_{0}^{T} f(u)\, dx)^{2}}+ |h(t)| \right\}\\
&\leq \frac{\lambda}{(c_{1}T)^{2}}  \int_0^t\frac{(t-s)^{2
\alpha-1}}{\Gamma(2 \alpha)} |f(u(s))-f(0)|ds
+ \int_0^t\frac{(t-s)^{2 \alpha-1}}{\Gamma(2 \alpha)}
\left(|h(s)|+\frac{\lambda}{(c_{1}T)^{2}}f(0)\right) ds \\
&\leq  \frac{\lambda L_{f}}{(c_{1}T)^{2}} \int_0^t\frac{(t-s)^{2
\alpha-1}}{\Gamma(2 \alpha)}  |u(s)| ds
+ \left(\frac{\lambda}{(c_{1}T)^{2}}f(0)+h_{\infty}\right)
\int_0^t\frac{(t-s)^{2 \alpha-1}}{\Gamma(2 \alpha)} ds.
\end{split}
\end{equation*}
Then,
\begin{equation*}
\begin{split}
e^{-Nt}|u(t)| &\leq  \frac{\lambda L_{f}}{(c_{1}T)^{2}}
\int_0^t\frac{(t-s)^{2 \alpha-1}}{\Gamma(2 \alpha)} e^{-Nt} |u(s)|
ds + ( \frac{\lambda}{(c_{1}T)^{2}}f(0)+h_{\infty})
\int_0^t\frac{(t-s)^{2 \alpha-1}}{\Gamma(2 \alpha)} e^{-Nt} ds \\
&\leq  \frac{\lambda L_{f}}{(c_{1}T)^{2}} \int_0^t\frac{(t-s)^{2
\alpha-1}}{\Gamma(2 \alpha)} e^{-N(t-s)}e^{-Ns} |u(s)| ds + (
\frac{\lambda}{(c_{1}T)^{2}}f(0)+h_{\infty})
\int_0^t\frac{(t-s)^{2 \alpha-1}}{\Gamma(2 \alpha)} e^{-N(t-s)}e^{-Ns}  ds \\
&\leq  \frac{\lambda L_{f}}{(c_{1}T)^{2}} \int_0^t\frac{(t-s)^{2
\alpha-1}}{\Gamma(2 \alpha)} e^{-N(t-s)}e^{-Ns} |u(s)| ds + (
\frac{\lambda}{(c_{1}T)^{2}}f(0)+h_{\infty})
\int_0^t\frac{(t-s)^{2 \alpha-1}}{\Gamma(2 \alpha)} e^{-N(t-s)}  ds \\
&\leq  \frac{\lambda L_{f}}{(c_{1}T)^{2}} \int_0^t\frac{(t-s)^{2
\alpha-1}}{\Gamma(2 \alpha)} e^{-N(t-s)}e^{-Ns} |u(s)| ds +
\frac{( \frac{\lambda}{(c_{1}T)^{2}}f(0)+h_{\infty})}{N^{2
\alpha}}
\int_0^{Nt}\frac{r^{2 \alpha-1}}{\Gamma(2 \alpha)} e^{-r}  dr \\
&\leq \frac{( \frac{\lambda}{(c_{1}T)^{2}}f(0)+h_{\infty})}{N^{2
\alpha}} + \frac{\lambda L_{f}}{(c_{1}T)^{2}}
\int_0^t\frac{(t-s)^{2 \alpha-1}}{\Gamma(2 \alpha)}
e^{-N(t-s)}\left(e^{-Ns} |u(s)|\right) ds.
\end{split}
\end{equation*}
Using Gronwall's lemma, we have
\begin{equation*}
\begin{split}
e^{-Nt}|u(t)| &\leq \frac{(
\frac{\lambda}{(c_{1}T)^{2}}f(0)+h_{\infty})}{N^{2 \alpha}}
e^{\frac{\lambda L_{f}}{(c_{1}T)^{2}}
\int_0^t\frac{(t-s)^{2 \alpha-1}}{\Gamma(2 \alpha)} e^{-N(t-s)} ds}\\
&\leq \frac{( \frac{\lambda}{(c_{1}T)^{2}}f(0)+h_{\infty})}{N^{2 \alpha}}
e^{\frac{\frac{\lambda L_{f}}{(c_{1}T)^{2}}}{N^{2 \alpha}}
\int_0^{Nt}\frac{r^{2 \alpha-1}}{\Gamma(2 \alpha)} e^{-r} dr }\\
&\leq\frac{( \frac{\lambda}{(c_{1}T)^{2}}f(0)+h_{\infty})}{N^{2
\alpha}} e^{\frac{\frac{\lambda L_{f}}{(c_{1}T)^{2}}}{N^{2 \alpha}}}.
\end{split}
\end{equation*}
Then,
\[
\|u \| \leq  \frac{(
\frac{\lambda}{(c_{1}T)^{2}}f(0)+h_{\infty})}{N^{2 \alpha}}
e^{\frac{\frac{\lambda L_{f}}{(c_{1}T)^{2}}}{N^{2 \alpha}}}=
\frac{( \frac{\lambda}{(c_{1}T)^{2}}f(0)+h_{\infty})}{N^{2
\alpha}} e^{\frac{\lambda L_{f}}{(c_{1}TN^{\alpha})^{2}}}
\]
and we conclude that $u$ is bounded.
\end{proof}

% -------------------------------------------------

\section*{Acknowledgements}

Work supported by {\it FEDER} funds through
{\it COMPETE} --- Operational Programme Factors of Competitiveness
(``Programa Operacional Factores de Competitividade'')
and by Portuguese funds through the
{\it Center for Research and Development
in Mathematics and Applications} (University of Aveiro)
and the Portuguese Foundation for Science and Technology
(``FCT --- Funda\c{c}\~{a}o para a Ci\^{e}ncia e a Tecnologia''),
within project PEst-C/MAT/UI4106/2011
with COMPETE number FCOMP-01-0124-FEDER-022690.
The authors were also supported by the project \emph{New Explorations in
Control Theory Through Advanced Research} (NECTAR) cofinanced by
FCT, Portugal, and the \emph{Centre National de la Recherche
Scientifique et Technique} (CNRST), Morocco.

The authors are grateful to an anonymous referee
for several corrections and useful comments.

% -------------------------------------------------

% -------------------------------------------------

\end{document}